\documentclass{amsart}
\usepackage{amsfonts}

\newtheorem{theorem}{Theorem}
\theoremstyle{plain}

\newtheorem{lemma}{Lemma}

\newtheorem{proposition}{Proposition}
\newtheorem{remark}{Remark}

\numberwithin{equation}{section}
\input{tcilatex}

\begin{document}
\title[Generalized Ponce's inequality]{Generalized Ponce's inequality}
\author{Julio Mu\~{n}oz}
\address{Departamento de Matem\'{a}ticas. Universidad de Castilla-La Mancha,
Toledo, Spain}
\email{julio.munoz@ulcm.es}
\date{July 10, 2019}

\begin{abstract}
This note provides the generalization of a remarkable inequality by A. C.
Ponce whose consequences are essential in several fields, as
Characterization of Sobolev Spaces or Nonlocal Modelization.
\end{abstract}

\maketitle

\section{Definitions and preliminaries}

Let $\Omega $ be an open bounded set in $\mathbb{R}^{N}$. We define the
family of kernels $\left( k_{\delta }\right) _{\delta >0}$ as a set of
radial, positive functions fulfilling the following properties:

\begin{enumerate}
\item 
\begin{equation*}
\frac{1}{C_{N}}\int_{B\left( 0,\delta \right) }k_{\delta }\left( \left\vert
s\right\vert \right) ds=1
\end{equation*}%
where 
\begin{equation*}
C_{N}=\frac{1}{\limfunc{meas}\left( S^{N-1}\right) }\int_{S^{N-1}}\left\vert
\sigma \cdot \mathbf{e}\right\vert ^{p}d\mathcal{H}^{N-1}\left( \sigma
\right) ,
\end{equation*}%
$\mathcal{H}^{N-1}$ stands for the $\left( N-1\right) $-dimensional
Haussdorff measure on the unit sphere $S^{N-1}$ and $\mathbf{e}$ is any unit
vector in $\mathbb{R}^{N}$ and $p>1.$ $B(0,\delta )$ is the notation for the
ball of center $0$ and radius $\delta .$

\item $\limfunc{supp}k_{\delta }\subset B\left( 0,\delta \right) $.
\end{enumerate}

We define the nonlocal operator $B_{h}$ in $L^{p}\left( \Omega \right)
\times L^{p}\left( \Omega \right) $ by means of the formula%
\begin{equation*}
B_{h}\left( u,u\right) =\int_{\Omega }\int_{\Omega }H\left( x^{\prime
},x\right) \frac{k_{\delta }\left( \left\vert x^{\prime }-x\right\vert
\right) }{\left\vert x^{\prime }-x\right\vert ^{p}}\left\vert u\left(
x^{\prime }\right) -u\left( x\right) \right\vert ^{p}dx^{\prime }dx,
\end{equation*}%
where $H\left( x^{\prime },x\right) =\frac{h\left( x^{\prime }\right)
+h\left( x\right) }{2},$ $h\in \mathcal{H}$, 
\begin{equation*}
\mathcal{H}\doteq \left\{ h:\Omega \rightarrow \mathbb{R}\mid h\left(
x\right) \in \lbrack h_{\min },h_{\max }]\text{ a.e. }x\in \Omega ,\text{ }%
h=0\text{ in }\mathbb{R}^{N}-\Omega \right\}
\end{equation*}%
and $0<h_{\min }<h_{\max }$ are given constants. \newline
If we choose $h=1$ the following compactness result it is well-known (see
for instance, \cite{brezis} and \cite[Proof of Theorem 1.2, p. 12]{ponce}):

\begin{theorem}
Assume $\left( u_{\delta }\right) _{\delta }$ is a sequence uniformly
bounded in $L^{p}\left( \Omega \right) $ and $C$ is a positive constant such
that 
\begin{equation}
\int_{\Omega }\int_{\Omega }\frac{k_{\delta }\left( \left\vert x^{\prime
}-x\right\vert \right) }{\left\vert x^{\prime }-x\right\vert ^{p}}\left\vert
u_{\delta }\left( x^{\prime }\right) -u_{\delta }\left( x\right) \right\vert
^{p}dx^{\prime }dx\leq C  \label{H1}
\end{equation}%
for any $\delta .$ Then, from $\left( u_{\delta }\right) _{\delta }$ we can
extract a subsequence, still denoted by $\left( u_{\delta }\right) _{\delta
},$ and we can find $u\in W^{1,p}\left( \Omega \right) $ such that, $%
u_{\delta }\rightarrow u$ strongly in $L^{p}\left( \Omega \right) $ if $%
\delta \rightarrow 0.$ Moreover%
\begin{equation}
\lim_{\delta \rightarrow 0}\int_{\Omega }\int_{\Omega }\frac{k_{\delta
}\left( \left\vert x^{\prime }-x\right\vert \right) }{\left\vert x^{\prime
}-x\right\vert ^{p}}\left\vert u_{\delta }\left( x^{\prime }\right)
-u_{\delta }\left( x\right) \right\vert ^{p}dx^{\prime }dx\geq \int_{\Omega
}\left\vert \nabla u\left( x\right) \right\vert ^{p}dx.  \label{Ponce}
\end{equation}
\end{theorem}

Even though several authors are involved in the proof, we shall refer to the
above estimation (\ref{Ponce}) as Ponce's inequality.

\subsection{Step 1: the objective}

Our goal is to prove the extension of (\ref{Ponce}) in the following sense: 
\begin{equation}
\lim_{\delta \rightarrow 0}\int_{\Omega }\int_{\Omega }H\left( x^{\prime
},x\right) \frac{k_{\delta }\left( \left\vert x^{\prime }-x\right\vert
\right) }{\left\vert x^{\prime }-x\right\vert ^{p}}\left\vert u_{\delta
}\left( x^{\prime }\right) -u_{\delta }\left( x\right) \right\vert
^{p}dx^{\prime }dx\geq \int_{\Omega }h\left( x\right) \left\vert \nabla
u\left( x\right) \right\vert ^{p}dx  \label{1b}
\end{equation}%
where $\Omega $ is an open bounded, $H\left( x^{\prime },x\right) =\frac{%
h\left( x^{\prime })+h(x\right) }{2}$ and $h\in \mathcal{H}.$\newline
As a corollary, we shall prove (\ref{Ponce}) for measurable sets, that is 
\begin{equation}
\lim_{\delta \rightarrow 0}\int_{G}\int_{G}\frac{k_{\delta }\left(
\left\vert x^{\prime }-x\right\vert \right) }{\left\vert x^{\prime
}-x\right\vert ^{p}}\left\vert u_{\delta }\left( x^{\prime }\right)
-u_{\delta }\left( x\right) \right\vert ^{p}dx^{\prime }dx\geq
\int_{G}\left\vert \nabla u\left( x\right) \right\vert ^{p}dx.
\label{Ponce_inequality}
\end{equation}%
\newline
where $G$ is any measurable set in $\Omega $.

\subsection{Motivation and organization of the paper}

The context in which we locate the present article is the study of the
nonlocal $p$-laplacian problem. Before proceeding, we precise of a little
bit of notation: we define the spaces $L_{0}^{p}\left( \Omega _{\delta
}\right) =\left\{ u\in L^{p}\left( \Omega _{\delta }\right) :u=0\text{ in }%
\mathbb{R}^{N}\setminus \Omega \right\} $ and $X=\left\{ u\in
L_{0}^{p}\left( \Omega _{\delta }\right) :B\left( u,u\right) <\infty
\right\} $ where $B=B_{1}$, that is, $B$ is the operator defined in $X\times
X$ by means of the formula 
\begin{equation*}
B\left( u,v\right) =\int_{\Omega _{\delta }}\int_{\Omega _{\delta }}\frac{%
k_{\delta }\left( \left\vert x^{\prime }-x\right\vert \right) }{\left\vert
x^{\prime }-x\right\vert ^{p}}\left\vert u\left( x^{\prime }\right) -u\left(
x\right) \right\vert ^{p-2}\left( u\left( x^{\prime }\right) -u\left(
x\right) \right) \left( v\left( x^{\prime }\right) -v\left( x\right) \right)
dx^{\prime }dx.
\end{equation*}%
We also define the space $X_{0}$ as $X_{0}=\func{cl}\left( C_{co}^{\infty
}\left( \Omega _{\delta }\right) \right) $ where%
\begin{equation*}
C_{co}^{\infty }\left( \Omega _{\delta }\right) =\left\{ f:\Omega _{\delta
}\rightarrow \mathbb{R}:f\in C_{c}^{\infty }\left( \Omega \right) \text{ and 
}f=0\text{ in }\Omega _{\delta }-\Omega \right\} \subset X
\end{equation*}%
and $\func{cl}\left( C_{co}^{\infty }\left( \Omega _{\delta }\right) \right) 
$ is the closure with respect to the norm $\left\Vert \cdot \right\Vert $
given in $X$ via the functional $B\left( \cdot ,\cdot \right) ,$ that means%
\begin{equation*}
X_{0}=\left\{ v\in X:\text{there is }\left( v_{j}\right) \subset
C_{co}^{\infty }\left( \Omega _{\delta }\right) \text{ such that }%
\lim_{j}B\left( v_{j}-v,v_{j}-v\right) =0\right\} .
\end{equation*}
We define now the following nonlocal variational problem: given $f\in
L^{p^{\prime }}\left( \Omega \right) $, where $p^{\prime }=\frac{p}{p-1}$,
and $p>1$, find $u\in X_{0}$ such that \ 
\begin{equation}
B_{h}\left( u,w\right) =\left( f,w\right) _{L^{p^{\prime }}\left( \Omega
\right) \times L^{p}\left( \Omega \right) }\text{ in }X_{0}.
\label{variational1}
\end{equation}%
Notice (\ref{variational1}) is equivalent to say that%
\begin{equation}
\int_{\Omega _{\delta }}\int_{\Omega _{\delta }}H\left( x^{\prime },x\right)
k_{\delta }\left( \left\vert x^{\prime }-x\right\vert \right) \frac{%
\left\vert u\left( x^{\prime }\right) -u\left( x\right) \right\vert
^{p-2}\left( u\left( x^{\prime }\right) -u\left( x\right) \right) \left(
w\left( x^{\prime }\right) -w\left( x\right) \right) }{\left\vert x^{\prime
}-x\right\vert ^{p}}dx^{\prime }dx=\int_{\Omega _{\delta }}fwdx
\label{variational2}
\end{equation}%
holds for any $w\in $ $X_{0}.$ Since the existence and uniqueness of
solution for this problem is a well-known fact, then, for $h$ fixed, and for
any $\delta ,$ there exists a solution $u_{\delta }.$ The aim is to check
whether the sequence of solutions $\left( u_{\delta }\right) _{\delta }$
converges to the solution of the corresponding local $p$-laplacian equation.
This convergence (or $G$-convergence$)$ clearly entails the study of the
minimization principle%
\begin{equation*}
\min_{w\in X_{0}}\left\{ \frac{1}{p}B_{h}\left( w,w\right) -\int_{\Omega
}f\left( x\right) w\left( x\right) dx\right\}
\end{equation*}%
and consequently, this task inevitably leads us to the study of the problem
posed above. \cite{Rossi2, D'Elia-Gunz2, Bonder, Andres-Julio3} are some
references where this type of convergence is analyzed.\smallskip

The manuscript is organized by means of three sections containing different
proofs of (\ref{1b}) and (\ref{Ponce_inequality}).

\section{First proof\label{first proof}}

Our essential tool in order to generalize (\ref{1b}), is a convenient Vitali
covering of the set $\Omega $ (see \cite{Saks} for the details).

\begin{lemma}
Let $\mathcal{A}=\left\{ F_{k}\right\} _{k\in K}$ be a Vitali covering of $%
\Omega .$ There is a sequence of $k_{i}\in K$ such that $\left\vert \Omega
\setminus \cup _{i}F_{k_{i}}\right\vert =0$ and the sets $F_{k_{i}}$ are
pairwise disjoints.
\end{lemma}

In a first step we assume $h$ is continuous a.e. in $\Omega $. We adapt \cite%
[Lemma 7.9, p. 129]{pedregal} in order to prove our key result:

\begin{proposition}
\label{propo}Let $\Omega \subset \mathbb{R}^{N}$ be an open bounded set in $%
\mathbb{R}^{N}$ such that $\left\vert \partial \Omega \right\vert =0$ and $%
f, $ a positive and a.e. continuos function defined in $\Omega .$ There
exists a set of points $\left\{ a_{ki}\right\} \subset \Omega $ and positive
sets of numbers $\left\{ \epsilon _{ki}\right\} \ $and $r_{k}\left(
a_{ki}\right) ,$ such that $\epsilon _{ki}\leq r_{k}\left( a_{ki}\right) ,$ 
\begin{eqnarray*}
&&\left\{ a_{ki}+\epsilon _{ki}\Omega \right\} \text{ are pairwise disjoint
for each }k, \\
\overline{\Omega } &=&\cup _{i}\left\{ a_{ki}+\epsilon _{ki}\overline{\Omega 
}\right\} \cup N_{k},\text{ where }\left\vert N_{k}\right\vert =0
\end{eqnarray*}%
and%
\begin{equation}
\int_{\Omega }f\left( x\right) \xi \left( x\right) dx=\sum_{i}f\left(
a_{ki}\right) \int_{a_{ki}+\epsilon _{ki}\Omega }\xi \left( x\right)
dx+o\left( 1\right)  \label{1}
\end{equation}%
for any $\xi \in L^{1}\left( \Omega \right) ,$ where $\left\vert o\left(
1\right) \right\vert \leq \frac{1}{k}\left\Vert \xi \right\Vert
_{L^{1}\left( \Omega \right) }$ if $k\rightarrow +\infty .$
\end{proposition}

\begin{proof}
Let $C$ be the set of points of continuity of $f$. We define the family of
sets 
\begin{equation*}
F_{k}=\left\{ a+\epsilon \overline{\Omega }:a\in C,\text{ }\epsilon \leq
r_{k}\left( a\right) ,\text{ }\left\vert f\left( x\right) -f\left( a\right)
\right\vert \leq \frac{1}{k}\text{ for any }x\in a+\epsilon \overline{\Omega 
}\text{ and }a+\epsilon \overline{\Omega }\subset \Omega \right\} .
\end{equation*}%
This family covers $C$ (and $\Omega )$ in the sense of Vitali. Thus, from
this family we are able to choose a numerable sequence of disjoints sets $%
\left\{ a_{ki}+\epsilon _{ki}\Omega \right\} _{i}\in F_{k},$ whose union
covers $\Omega .$ Since $f$ is continuous in $\overline{a_{ki}+\epsilon
_{ki}\Omega }$, the sequence $\epsilon _{ki}$ can be chosen so that%
\begin{equation*}
\left\vert f\left( x\right) -f\left( a_{ki}\right) \right\vert \leq \frac{1}{%
k},\text{ for any }x\in \overline{a_{ki}+\epsilon _{ki}\Omega }
\end{equation*}%
for any $i$ and any $k.$ Consequently, we note%
\begin{eqnarray*}
&&\left\vert \int_{\Omega }\xi \left( x\right) f\left( x\right)
dx-\sum_{i}f\left( a_{ki}\right) \int_{a_{ki}+\epsilon _{ki}\Omega }\xi
\left( x\right) dx\right\vert \\
&=&\left\vert \sum_{i}\int_{a_{ki}+\epsilon _{ki}\Omega }\left( f\left(
x\right) -f\left( a_{ki}\right) \right) \xi \left( x\right) dx\right\vert \\
&\leq &\sum_{i}\int_{a_{ki}+\epsilon _{ki}\Omega }\left\vert \left( f\left(
x\right) -f\left( a_{ki}\right) \right) \right\vert \left\vert \xi \left(
x\right) \right\vert dx \\
&\leq &\frac{1}{k}\sum_{i}\int_{a_{ki}+\epsilon _{ki}\Omega }\left\vert \xi
\left( x\right) \right\vert dx \\
&=&\frac{1}{k}\left\Vert \xi \right\Vert _{L^{1}\left( \Omega \right) }
\end{eqnarray*}
\end{proof}

\subsection{Application}

We apply the above analysis to the integral 
\begin{equation*}
I=\int_{\Omega }\int_{\Omega }H\left( x^{\prime },x\right) \xi _{\delta
}\left( x^{\prime },x\right) dx^{\prime }dx
\end{equation*}%
where%
\begin{equation}
\xi _{\delta }\left( x^{\prime },x\right) =\frac{k_{\delta }\left(
\left\vert x^{\prime }-x\right\vert \right) }{\left\vert x^{\prime
}-x\right\vert ^{p}}\left\vert u_{\delta }\left( x^{\prime }\right)
-u_{\delta }\left( x\right) \right\vert  \label{xis}
\end{equation}%
We consider $\Omega \times \Omega $ instead of $\Omega $ and now, $f\left(
x^{\prime },x\right) $ is the symmetric function $H\left( x^{\prime
},x\right) =\frac{h\left( x^{\prime }\right) +h\left( x\right) }{2},$ with $%
h\in \mathcal{H}.$ We assume $h$\ is continuous and we notice the family $%
\cup _{i,j}\left( a_{ki}+\epsilon _{ki}\Omega \right) \times \left(
a_{kj}+\epsilon _{kj}\Omega \right) $ is a Vitali covering of $\Omega \times
\Omega .$ Then, according to the above discussion 
\begin{eqnarray*}
I &=&\sum_{i,j}H\left( a_{ki},a_{kj}\right) \int_{a_{ki}+\epsilon
_{ki}\Omega }\int_{a_{kj}+\epsilon _{kj}\Omega }\frac{k_{\delta }\left(
\left\vert x^{\prime }-x\right\vert \right) }{\left\vert x^{\prime
}-x\right\vert ^{p}}\left\vert u_{\delta }\left( x^{\prime }\right)
-u_{\delta }\left( x\right) \right\vert ^{p}dx^{\prime }dx+o\left( 1\right)
\\
&\geq &\sum_{i}H\left( a_{ki},a_{ki}\right) \int_{a_{ki}+\epsilon
_{ki}\Omega }\int_{a_{ki}+\epsilon _{ki}\Omega }\frac{k_{\delta }\left(
\left\vert x^{\prime }-x\right\vert \right) }{\left\vert x^{\prime
}-x\right\vert ^{p}}\left\vert u_{\delta }\left( x^{\prime }\right)
-u_{\delta }\left( x\right) \right\vert ^{p}dx^{\prime }dx+o\left( 1\right)
\\
&=&\sum_{i}h\left( a_{ki}\right) \int_{a_{ki}+\epsilon _{ki}\Omega
}\int_{a_{ki}+\epsilon _{ki}\Omega }\frac{k_{\delta }\left( \left\vert
x^{\prime }-x\right\vert \right) }{\left\vert x^{\prime }-x\right\vert ^{p}}%
\left\vert u_{\delta }\left( x^{\prime }\right) -u_{\delta }\left( x\right)
\right\vert ^{p}dx^{\prime }dx+o\left( 1\right)
\end{eqnarray*}

We pass to the limit when $\delta \rightarrow 0$ in $I:$ we use (\ref{H1}),
Fatou's Lemma and (\ref{Ponce}) for open sets, to derive 
\begin{eqnarray*}
\lim_{\delta \rightarrow 0}I &\geq &\lim_{\delta \rightarrow
0}\sum_{i}h\left( a_{ki}\right) \int_{a_{ki}+\epsilon _{ki}\Omega
}\int_{a_{ki}+\epsilon _{ki}\Omega }\frac{k_{\delta }\left( \left\vert
x^{\prime }-x\right\vert \right) }{\left\vert x^{\prime }-x\right\vert ^{p}}%
\left\vert u_{\delta }\left( x^{\prime }\right) -u_{\delta }\left( x\right)
\right\vert ^{p}dx^{\prime }dx-\frac{C}{k} \\
&\geq &\sum_{i}h\left( a_{ki}\right) \left( \lim_{\delta \rightarrow
0}\int_{a_{ki}+\epsilon _{ki}\Omega }\int_{a_{ki}+\epsilon _{ki}\Omega }%
\frac{k_{\delta }\left( \left\vert x^{\prime }-x\right\vert \right) }{%
\left\vert x^{\prime }-x\right\vert ^{p}}\left\vert u_{\delta }\left(
x^{\prime }\right) -u_{\delta }\left( x\right) \right\vert ^{p}dx^{\prime
}dx\right) -\frac{C}{k} \\
&\geq &\sum_{i}h\left( a_{ki}\right) \left( \int_{a_{ki}+\epsilon
_{ki}\Omega }\left\vert \nabla u\left( x\right) \right\vert ^{p}dx\right) -%
\frac{C}{k}.
\end{eqnarray*}

If we take limits in $k\rightarrow +\infty ,$ then the above estimation gives%
\begin{equation*}
\lim_{\delta \rightarrow 0}I\geq \lim_{k\rightarrow +\infty }\sum_{i}h\left(
a_{ki}\right) \int_{a_{ki}+\epsilon _{ki}\Omega }\left\vert \nabla u\left(
x\right) \right\vert ^{p}dx.
\end{equation*}%
By using again Proposition \ref{propo}, the last inequality is clearly
equivalent to 
\begin{equation*}
\lim_{\delta \rightarrow 0}\int_{\Omega }\int_{\Omega }H\left( x^{\prime
},x\right) \frac{k_{\delta }\left( \left\vert x^{\prime }-x\right\vert
\right) }{\left\vert x^{\prime }-x\right\vert ^{p}}\left\vert u_{\delta
}\left( x^{\prime }\right) -u_{\delta }\left( x\right) \right\vert
^{p}dx^{\prime }dx\geq \int_{\Omega }h\left( x\right) \left\vert \nabla
u\left( x\right) \right\vert ^{p}dx
\end{equation*}%
which is the thesis (\ref{1b}) we desired to prove.

\begin{remark}
\label{nota3}The analysis and conclusion we have just arrived, remain valid
if we consider any open set $O\subset \Omega $ such that $\left\vert
\partial O\right\vert =0.$ We can go an step further, the inequality%
\begin{equation}
\lim_{\delta \rightarrow 0}\int_{\Omega }\int_{\Omega }F\left( x^{\prime
},x\right) \frac{k_{\delta }\left( \left\vert x^{\prime }-x\right\vert
\right) }{\left\vert x^{\prime }-x\right\vert ^{p}}\left\vert u_{\delta
}\left( x^{\prime }\right) -u_{\delta }\left( x\right) \right\vert
^{p}dx^{\prime }dx\geq \int_{\Omega }F\left( x,x\right) \left\vert \nabla
u\left( x\right) \right\vert ^{p}dx  \label{4}
\end{equation}%
holds for any symmetric, positive and continuous function $F\in L^{\infty
}\left( \Omega \times \Omega \right) .$
\end{remark}

\subsection{Extension to the case of measurable functions\label{measurable}}

Assume now $h$ is just measurable. We know $\limfunc{supp}H\subset \Omega
\times \Omega $ and $H=0$ otherwise. By Luzin's Theorem (see \cite[Theorem
2.24, p. 62]{Rudin}), given an arbitrary $\epsilon >0$ there exists a
continuous function $G\in C_{c}\left( \Omega \times \Omega \right) $ such
that $\sup G\left( x,y\right) \leq \sup H\left( x,y\right) $ and $G\left(
x,y\right) =H\left( x,y\right) $ for any $\left( x,y\right) \in \left(
\Omega \times \Omega \right) \setminus \mathcal{E},$ where $\mathcal{E}$ is
a measurable set such that $\left\vert \mathcal{E}\right\vert <\epsilon
^{2}. $ Since $H$ is symmetric then we are allowed to assume $\left( \Omega
\times \Omega \right) \setminus \mathcal{E=}\left( \Omega -E\right) \times
\left( \Omega -E\right) $ where $E\subset \Omega $ is a measurable set such
that $\left\vert E\right\vert <\epsilon .$

At this stage we consider a family of relative open sets $B_{n}$ in $\Omega
, $ such that $E\subset \overline{B}_{n}\subset \Omega $ and $%
B_{n}\downarrow E.$ Then%
\begin{eqnarray*}
&&\int_{\Omega }\int_{\Omega }H\left( x^{\prime },x\right) \frac{k_{\delta
}\left( \left\vert x^{\prime }-x\right\vert \right) }{\left\vert x^{\prime
}-x\right\vert ^{p}}\left\vert u_{\delta }\left( x^{\prime }\right)
-u_{\delta }\left( x\right) \right\vert ^{p}dx^{\prime }dx \\
&\geq &\iint_{\left( \Omega -\overline{B}_{n}\right) \times \left( \Omega -%
\overline{B}_{n}\right) }H\left( x^{\prime },x\right) \frac{k_{\delta
}\left( \left\vert x^{\prime }-x\right\vert \right) }{\left\vert x^{\prime
}-x\right\vert ^{p}}\left\vert u_{\delta }\left( x^{\prime }\right)
-u_{\delta }\left( x\right) \right\vert ^{p}dx^{\prime }dx \\
&=&\iint_{\left( \Omega -\overline{B}_{n}\right) \times \left( \Omega -%
\overline{B}_{n}\right) }G\left( x^{\prime },x\right) \frac{k_{\delta
}\left( \left\vert x^{\prime }-x\right\vert \right) }{\left\vert x^{\prime
}-x\right\vert ^{p}}\left\vert u_{\delta }\left( x^{\prime }\right)
-u_{\delta }\left( x\right) \right\vert ^{p}dx^{\prime }dx
\end{eqnarray*}%
We fix $n$ and take limits in $\delta $ to get%
\begin{eqnarray*}
&&\lim_{\delta \rightarrow 0}\int_{\Omega }\int_{\Omega }H\left( x^{\prime
},x\right) \frac{k_{\delta }\left( \left\vert x^{\prime }-x\right\vert
\right) }{\left\vert x^{\prime }-x\right\vert ^{p}}\left\vert u_{\delta
}\left( x^{\prime }\right) -u_{\delta }\left( x\right) \right\vert
^{p}dx^{\prime }dx \\
&\geq &\lim_{\delta \rightarrow 0}\iint_{\left( \Omega -\overline{B}%
_{n}\right) \times \left( \Omega -\overline{B}_{n}\right) }G\left( x^{\prime
},x\right) \frac{k_{\delta }\left( \left\vert x^{\prime }-x\right\vert
\right) }{\left\vert x^{\prime }-x\right\vert ^{p}}\left\vert u_{\delta
}\left( x^{\prime }\right) -u_{\delta }\left( x\right) \right\vert
^{p}dx^{\prime }dx \\
&\geq &\int_{\left( \Omega -\overline{B}_{n}\right) }G\left( x,x\right)
\left\vert \nabla u\left( x\right) \right\vert ^{p}dx \\
&=&\int_{\left( \Omega -\overline{B}_{n}\right) }h\left( x\right) \left\vert
\nabla u\left( x\right) \right\vert ^{p}dx
\end{eqnarray*}%
where the second inequality is true thanks to (\ref{4}). Then, since $%
B_{n}\downarrow E,$ we obtain 
\begin{eqnarray*}
&&\lim_{\delta \rightarrow 0}\int_{\Omega }\int_{\Omega }H\left( x^{\prime
},x\right) \frac{k_{\delta }\left( \left\vert x^{\prime }-x\right\vert
\right) }{\left\vert x^{\prime }-x\right\vert ^{p}}\left\vert u_{\delta
}\left( x^{\prime }\right) -u_{\delta }\left( x\right) \right\vert
^{p}dx^{\prime }dx \\
&\geq &\int_{\Omega }h\left( x\right) \left\vert \nabla u\left( x\right)
\right\vert ^{p}dx-\int_{E}h\left( x\right) \left\vert \nabla u\left(
x\right) \right\vert ^{p}dx
\end{eqnarray*}%
By letting $\epsilon \downarrow 0$ and using $\left\vert E\right\vert \leq
\epsilon $ we obtain (\ref{1b}):%
\begin{equation}
\lim_{\delta \rightarrow 0}\int_{\Omega }\int_{\Omega }H\left( x^{\prime
},x\right) \frac{k_{\delta }\left( \left\vert x^{\prime }-x\right\vert
\right) }{\left\vert x^{\prime }-x\right\vert ^{p}}\left\vert u_{\delta
}\left( x^{\prime }\right) -u_{\delta }\left( x\right) \right\vert
^{p}dx^{\prime }dx\geq \int_{\Omega }H\left( x,x\right) \left\vert \nabla
u\left( x\right) \right\vert ^{p}dx.  \label{5}
\end{equation}

Finally, in order to avoid the assumption $\left\vert \partial \Omega
\right\vert =0$ we simplify as follows: for any given $\Omega $ we consider $%
\Omega _{r},\ $with $r>0$, and we extend $H$ by zero in $\Omega _{r}\times
\Omega _{r}\setminus \Omega \times \Omega .$ If we denote this extended
function by $H_{0},$ which is measurable, and we take into account that
boundary of $\Omega _{r}$ has measure zero, then (\ref{5}) allow us to write 
\begin{equation*}
\lim_{\delta \rightarrow 0}\int_{\Omega _{r}}\int_{\Omega _{r}}H_{0}\left(
x^{\prime },x\right) \frac{k_{\delta }\left( \left\vert x^{\prime
}-x\right\vert \right) }{\left\vert x^{\prime }-x\right\vert ^{p}}\left\vert
u_{\delta }\left( x^{\prime }\right) -u_{\delta }\left( x\right) \right\vert
^{p}dx^{\prime }dx\geq \int_{\Omega _{r}}H_{0}\left( x,x\right) \left\vert
\nabla u\left( x\right) \right\vert ^{p}dx.
\end{equation*}%
But the above inequality coincides with (\ref{5}), 
\begin{equation}
\lim_{\delta \rightarrow 0}\int_{\Omega }\int_{\Omega }H\left( x^{\prime
},x\right) \frac{k_{\delta }\left( \left\vert x^{\prime }-x\right\vert
\right) }{\left\vert x^{\prime }-x\right\vert ^{p}}\left\vert u_{\delta
}\left( x^{\prime }\right) -u_{\delta }\left( x\right) \right\vert
^{p}dx^{\prime }dx\geq \int_{\Omega }H\left( x,x\right) \left\vert \nabla
u\left( x\right) \right\vert ^{p}dx  \label{6}
\end{equation}%
for any open and bounded set $\Omega .$

\subsection{A Corollary}

We apply (\ref{5}) to the case $F\left( x^{\prime },x\right) =I_{G\times
G}\left( x^{\prime },x\right) $, where $G$ is any measurable set included in 
$\Omega $: on the one hand, (\ref{6}) guarantees%
\begin{equation*}
\begin{tabular}{c}
$\displaystyle\lim_{\delta \rightarrow 0}\int_{\Omega }\int_{\Omega }F\left(
x^{\prime },x\right) \frac{k_{\delta }\left( \left\vert x^{\prime
}-x\right\vert \right) }{\left\vert x^{\prime }-x\right\vert ^{p}}\left\vert
u_{\delta }\left( x^{\prime }\right) -u_{\delta }\left( x\right) \right\vert
^{p}dx^{\prime }dx\geq \int_{\Omega }F\left( x,x\right) \left\vert \nabla
u\left( x\right) \right\vert ^{p}dx\smallskip $ \\ 
$\displaystyle=\int_{G}I_{G}\left( x\right) \left\vert \nabla u\left(
x\right) \right\vert ^{p}dx=\int_{G}\left\vert \nabla u\left( x\right)
\right\vert ^{p}dx.$%
\end{tabular}%
\end{equation*}

On the other hand, it is obvious that

\begin{equation*}
\int_{\Omega }\int_{\Omega }F\left( x^{\prime },x\right) \frac{k_{\delta
}\left( \left\vert x^{\prime }-x\right\vert \right) }{\left\vert x^{\prime
}-x\right\vert ^{p}}\left\vert u_{\delta }\left( x^{\prime }\right)
-u_{\delta }\left( x\right) \right\vert ^{p}dx^{\prime }dx=\int_{G}\int_{G}%
\frac{k_{\delta }\left( \left\vert x^{\prime }-x\right\vert \right) }{%
\left\vert x^{\prime }-x\right\vert ^{p}}\left\vert u_{\delta }\left(
x^{\prime }\right) -u_{\delta }\left( x\right) \right\vert ^{p}dx^{\prime
}dx.
\end{equation*}%
Consequently (\ref{Ponce_inequality}) has been proved for any measurable set 
$G\subset \Omega $.

\section{A second proof\label{second proof}}

We firstly prove (\ref{Ponce_inequality}) and then (\ref{1b}). By
hypothesis, there is constant $C$ such that $\int_{\Omega }\int_{\Omega }\xi
_{\delta }\left( x^{\prime },x\right) dx^{\prime }dx\leq C$ for any $\delta
, $ where $\xi _{\delta }\left( x^{\prime },x\right) \ $is defined as in (%
\ref{xis}). Thus $\left( \xi _{\delta }\right) _{\delta }$ is a sequence
uniformly bounded in $L^{1}\left( \Omega \times \Omega \right) \ $and under
these circumstances, we can use Chacon's biting Lemma (\cite{Chacon}) to
ensure the existence of a decreasing sequence of measurable sets $\mathcal{E}%
_{n}\subset \Omega \times \Omega ,$ such that $\left\vert \mathcal{E}%
_{n}\right\vert \downarrow 0,\ $and a function $\phi \in L^{1}\left( \Omega
\times \Omega \right) ,$ such that $\xi _{\delta }\rightharpoonup \xi $
weakly in $L^{1}\left( \Omega \times \Omega \setminus \mathcal{E}_{n}\right) 
$ for all $n.$ Since we are dealing with a sequence of symmetric functions
we can ensure $\Omega \times \Omega \setminus \mathcal{E}_{n}=\left( \Omega
\setminus E_{n}\right) \times \left( \Omega \setminus E_{n}\right) $ where
the sequence of sets $E_{n}\subset \Omega $ \ is decreasing and $\left\vert
E_{n}\right\vert \downarrow 0$ if $n\rightarrow \infty .$

Let $B_{n}$ be any open set such that $E_{n}\subset B_{n},$ and $\left\vert
B_{n}\right\vert \downarrow 0.$ We apply Chacon's biting lemma to guarantee
the convergence%
\begin{equation*}
\lim_{\delta \rightarrow 0}\iint_{A\times A}\xi _{\delta }\left( x^{\prime
},x\right) dx^{\prime }dx=\iint_{A\times A}\xi \left( x^{\prime },x\right)
dx^{\prime }dx
\end{equation*}%
for any open $A\times A\subset \left( \Omega \setminus B_{n}\right) \times
\left( \Omega \setminus B_{n}\right) .$ Also, (\ref{Ponce_inequality}) for
open sets gives%
\begin{equation*}
\lim_{\delta \rightarrow 0}\iint_{A\times A}\xi _{\delta }\left( x^{\prime
},x\right) dx^{\prime }dx\geq \int_{A}\left\vert \nabla u\left( x\right)
\right\vert ^{p}dx,
\end{equation*}%
for any open set $A\subset \Omega \setminus B_{n}.$ Thus, the above
discussion gives%
\begin{equation*}
\iint_{A\times A}\xi \left( x^{\prime },x\right) dx^{\prime }dx\geq
\int_{A}\left\vert \nabla u\left( x\right) \right\vert ^{p}dx\text{ for any
open set }A\subset \Omega \setminus B_{n}.
\end{equation*}%
If this statement is true for open sets $A\subset \Omega \setminus B_{n}$,
it is so for measurable sets $E\subset \Omega \setminus B_{n}.$ \newline
We analyze $\lim_{\delta \rightarrow 0}\iint_{G\times G}\xi _{\delta }\left(
x^{\prime },x\right) dx^{\prime }dx:$ we note%
\begin{equation*}
\iint_{G\times G}\xi _{\delta }\left( x^{\prime },x\right) dx^{\prime
}dx\geq \iint_{\left( G\setminus \overline{B}_{n}\right) \times \left(
G\setminus \overline{B}_{n}\right) }\xi _{\delta }\left( x^{\prime
},x\right) dx^{\prime }dx
\end{equation*}%
which, thanks to Chacon's biting lemma, provides the estimation%
\begin{equation*}
\lim_{\delta \rightarrow 0}\iint_{G\times G}\xi _{\delta }\left( x^{\prime
},x\right) dx^{\prime }dx\geq \iint_{\left( G\setminus \overline{B}%
_{n}\right) \times \left( G\setminus \overline{B}_{n}\right) }\xi \left(
x^{\prime },x\right) dx^{\prime }dx
\end{equation*}%
Since $G\setminus \overline{B}_{n}$ is a measurable set included in $\Omega
\setminus B_{n},$ then we have the estimation 
\begin{equation*}
\iint_{\left( G\setminus \overline{B}_{n}\right) \times \left( G\setminus 
\overline{B}_{n}\right) }\xi \left( x^{\prime },x\right) dx^{\prime }dx\geq
\int_{G\setminus \overline{B}_{n}}\left\vert \nabla u\left( x\right)
\right\vert ^{p}dx,
\end{equation*}%
which implies%
\begin{equation*}
\lim_{\delta \rightarrow 0}\iint_{G\times G}\xi _{\delta }\left( x^{\prime
},x\right) dx^{\prime }dx\geq \int_{G\setminus \overline{B}_{n}}\left\vert
\nabla u\left( x\right) \right\vert ^{p}dx.
\end{equation*}%
By letting $n\rightarrow \infty $ we\ finish the proof of (\ref%
{Ponce_inequality}).

\subsection{A corollary}

Assume $h$ is a given simple function defined in $\Omega .$ Then $h$ can be
written as $h\left( x\right) =\sum_{i=1}^{m}h_{i}I_{B_{i}}\left( x\right) $,
where $\left\{ B_{i}\right\} $ is a finite covering of disjoint measurable
sets of $\Omega $ and $\left( h_{i}\right) _{i}$ is a set of numbers such
that $h_{\min }\leq h_{i}\leq h_{\max }$. Consequently, it can be easily
checked that%
\begin{equation*}
I\doteq \int_{\Omega }\int_{\Omega }H\left( x^{\prime },x\right) k_{\delta
}\left( \left\vert x^{\prime }-x\right\vert \right) \frac{\left\vert
u_{\delta }\left( x^{\prime }\right) -u_{\delta }\left( x\right) \right\vert
^{p}}{\left\vert x^{\prime }-x\right\vert ^{p}}dx^{\prime }dx\geq
\sum_{i=1}^{m}h_{i}\int_{B_{i}}\int_{B_{i}}k_{\delta }\left( \left\vert
x^{\prime }-x\right\vert \right) \frac{\left\vert u_{\delta }\left(
x^{\prime }\right) -u_{\delta }\left( x\right) \right\vert ^{p}}{\left\vert
x^{\prime }-x\right\vert ^{p}}dx^{\prime }dx.
\end{equation*}%
If we use the inequality (\ref{Ponce_inequality}) for measurable sets that
we have just proved, we straightforwardly infer 
\begin{equation*}
\lim_{\delta \rightarrow 0}I\geq \sum_{i=1}^{m}h_{i}\int_{B_{i}}\left\vert
\nabla u\left( x\right) \right\vert ^{p}dx=\int_{\Omega }h\left( x\right)
\left\vert \nabla u\left( x\right) \right\vert ^{p}dx.
\end{equation*}%
Let $h$ be a measurable function. By recalling that any measurable function $%
h$ can be pointwise approximated by $\left( s_{n}\right) _{n},$ an
increasing sequence of simple functions, we are allow to write 
\begin{align*}
& \lim_{\delta \rightarrow 0}\int_{\Omega }\int_{\Omega }H\left( x^{\prime
},x\right) k_{\delta }\left( \left\vert x^{\prime }-x\right\vert \right) 
\frac{\left\vert u_{\delta }\left( x^{\prime }\right) -u_{\delta }\left(
x\right) \right\vert ^{p}}{\left\vert x^{\prime }-x\right\vert ^{p}}%
dx^{\prime }dx \\
& =\lim_{\delta \rightarrow 0}\int_{\Omega }h\left( x\right) \int_{\Omega
}k_{\delta }\left( \left\vert x^{\prime }-x\right\vert \right) \frac{%
\left\vert u_{\delta }\left( x^{\prime }\right) -u_{\delta }\left( x\right)
\right\vert ^{p}}{\left\vert x^{\prime }-x\right\vert ^{p}}dx^{\prime }dx \\
& \geq \lim_{\delta \rightarrow 0}\int_{\Omega _{\delta }}s_{n}\left(
x\right) \int_{\Omega _{\delta }}k_{\delta }\left( \left\vert x^{\prime
}-x\right\vert \right) \frac{\left\vert u_{\delta }\left( x^{\prime }\right)
-u_{\delta }\left( x\right) \right\vert ^{p}}{\left\vert x^{\prime
}-x\right\vert ^{p}}dx^{\prime }dx \\
& \geq \int_{\Omega }s_{n}\left( x\right) \left\vert \nabla u\left( x\right)
\right\vert ^{p}dx.
\end{align*}%
It suffices to take limits in $n$ and apply the monotone convergence Theorem
to establish (\ref{1b}).

\section{A third proof}

The idea is to reproduce the arguments from \cite{ponce}. In a first step we
assume $h:\overline{\Omega }\rightarrow \left[ h_{\min },h_{\max }\right] $
is a continuous function. Moreover, without loss of generality, $h$ is
supposed to be a continuous function in the set $\Omega _{s}=\Omega \cup
\left\{ \cup _{p\in \partial \Omega }B\left( p,s\right) \right\} ,$ where $s$
is a fixed positive number.

Now, for the proof of (\ref{1b}) the key idea is to extend the Stein
inequality (see \cite[Lemma 4, p. 245]{ponce2}) in the following sense: by
using Jensen's inequality and performing a change of variables, we deduce
the inequality%
\begin{equation*}
\int_{\Omega }\int_{\Omega }H_{r}\left( x^{\prime },x\right) k_{\delta
}\left( \left\vert x^{\prime }-x\right\vert \right) \frac{\left\vert
u_{\delta }\left( x^{\prime }\right) -u_{\delta }\left( x\right) \right\vert
^{p}}{\left\vert x^{\prime }-x\right\vert ^{p}}dx^{\prime }dx\geq
\int_{\Omega _{-r}}\int_{\Omega _{-r}}H\left( x^{\prime },x\right) k_{\delta
}\left( \left\vert x^{\prime }-x\right\vert \right) \frac{\left\vert
u_{r,\delta }\left( x^{\prime }\right) -u_{r,\delta }\left( x\right)
\right\vert ^{p}}{\left\vert x^{\prime }-x\right\vert ^{p}}dx^{\prime }dx
\end{equation*}%
for any $\delta <r,$ where $u_{r,\delta }=\eta _{r}\ast u_{\delta },$ $\eta
_{r}\left( x\right) =\frac{1}{r^{N}}\eta \left( \frac{x}{r}\right) ,$ $x\in 
\mathbb{R}^{N},$ $\eta \in ,$ $\eta \ $is a nonnegative and radial function
from $C_{c}^{\infty }\left( B\left( 0,1\right) \right) $ such that $\int
\eta \left( x\right) dx=1,$ 
\begin{equation*}
H_{r}\left( x^{\prime },x\right) =\frac{\left( \eta _{r}\ast h\right) \left(
x^{\prime }\right) +\left( \eta _{r}\ast h\right) \left( x\right) }{2}
\end{equation*}%
and $\Omega _{-r}=\left\{ x\in \Omega :\limfunc{dist}(x,\partial \Omega
)>r\right\} .$ Due to the continuity of $H$ in $\Omega _{s}\times \Omega
_{s} $ we know $H_{r}\left( x^{\prime },x\right) \rightarrow H\left(
x^{\prime },x\right) $ uniformly on compact sets of $\Omega _{s}\times
\Omega _{s},$ whereby, for any $\epsilon >0,$ we can choose $r_{0}>0$ such
that%
\begin{equation*}
\left\vert \int_{\Omega }\int_{\Omega }\left( H\left( x^{\prime },x\right)
-H_{r}\left( x^{\prime },x\right) \right) k_{\delta }\left( \left\vert
x^{\prime }-x\right\vert \right) \frac{\left\vert u_{\delta }\left(
x^{\prime }\right) -u_{\delta }\left( x\right) \right\vert ^{p}}{\left\vert
x^{\prime }-x\right\vert ^{p}}dx^{\prime }dx\right\vert \leq \epsilon C\text{
}
\end{equation*}%
for any $r<r_{0}$ and uniformly in $\delta >0.$ Then 
\begin{eqnarray*}
&&\lim_{\delta \rightarrow 0}\int_{\Omega }\int_{\Omega }H\left( x^{\prime
},x\right) k_{\delta }\left( \left\vert x^{\prime }-x\right\vert \right) 
\frac{\left\vert u_{\delta }\left( x^{\prime }\right) -u_{\delta }\left(
x\right) \right\vert ^{p}}{\left\vert x^{\prime }-x\right\vert ^{p}}%
dx^{\prime }dx \\
&\geq &\lim_{\delta \rightarrow 0}\int_{\Omega _{-r}}\int_{\Omega
_{-r}}H\left( x^{\prime },x\right) k_{\delta }\left( \left\vert x^{\prime
}-x\right\vert \right) \frac{\left\vert u_{r,\delta }\left( x^{\prime
}\right) -u_{r,\delta }\left( x\right) \right\vert ^{p}}{\left\vert
x^{\prime }-x\right\vert ^{p}}dx^{\prime }dx-\epsilon C
\end{eqnarray*}%
for any $r<r_{0}.$ At this point we notice Proposition 1 from \cite[p. 242]%
{ponce2} can be modified by including the term $H\left( x^{\prime },x\right) 
$ within the integrand: this is factually what Remark \ref{nota3}
establishes. Then, if we pass to the limit in $\delta \rightarrow 0$ and we
use the convergence of $\rho _{r}\ast u_{\delta }\rightarrow \rho _{r}\ast u$
in $C^{2}\left( \overline{\Omega }_{-r}\right) ,$ we get 
\begin{equation*}
\lim_{\delta \rightarrow 0}\int_{\Omega _{-r}}\int_{\Omega _{-r}}H\left(
x^{\prime },x\right) k_{\delta }\left( \left\vert x^{\prime }-x\right\vert
\right) \frac{\left\vert u_{r,\delta }\left( x^{\prime }\right) -u_{r,\delta
}\left( x\right) \right\vert ^{p}}{\left\vert x^{\prime }-x\right\vert ^{p}}%
dx^{\prime }dx\geq \int_{\Omega _{-r}}h\left( x\right) \left\vert \nabla
\left( \rho _{r}\ast u\right) \left( x\right) \right\vert ^{p}dx^{\prime }dx.
\end{equation*}%
Consequently, by letting $r\rightarrow 0$\ in the inequality from above, and
taking into account that $\nabla \left( \rho _{r}\ast u\right) $ strongly
converges to $\nabla u$ in $L^{p}\left( \Omega \right) ,$ we derive%
\begin{equation*}
\lim_{\delta \rightarrow 0}\int_{\Omega }\int_{\Omega }H\left( x^{\prime
},x\right) k_{\delta }\left( \left\vert x^{\prime }-x\right\vert \right) 
\frac{\left\vert u_{\delta }\left( x^{\prime }\right) -u_{\delta }\left(
x\right) \right\vert ^{p}}{\left\vert x^{\prime }-x\right\vert ^{p}}%
dx^{\prime }dx\geq \int_{\Omega }h\left( x\right) \left\vert \nabla u\left(
x\right) \right\vert ^{p}dx^{\prime }dx-\epsilon C.
\end{equation*}%
Now, since $\epsilon $ is arbitrarily small, then the thesis is proved under
the\ assumption that $h$ is continuous in $\Omega _{s}$.

If $h:\Omega \rightarrow \left[ h_{\min },h_{\max }\right] $ is a measurable
function, then we extend it by zero to $\Omega _{s}$ and then we apply
Luzin's Theorem to this extended function. The remain of the details follows
along the same lines of Subsection \ref{measurable}.

\section{Acknowledgement}

This work was supported by Spanish Projects MTM2017-87912-P,
SBPLY/17/180501/000452 and by Universidad de Castilla-La Mancha Support for
Groups 2018.

\end{document}